\documentclass[11pt, a4paper]{article}
\usepackage{amssymb,amsmath,amsthm}
\usepackage{fancyhdr}
\usepackage[british]{babel}
\usepackage{enumitem}

\usepackage[usenames,dvipsnames,table]{xcolor}
\usepackage{tikz}
\usetikzlibrary{decorations.pathreplacing}
\usetikzlibrary{decorations.pathmorphing}
\usepackage{bbm}

\usepackage{kotex}
\usepackage{todonotes}
\usepackage{subcaption}

\usepackage{hyperref}
\usepackage[margin=2.5cm]{geometry}

\newtheorem{theorem}{Theorem}[section]
\newtheorem{prop}[theorem]{Proposition}
\newtheorem{lemma}[theorem]{Lemma}

\newtheorem{question}[theorem]{Question}

\theoremstyle{definition}

\newtheorem{defin}[theorem]{Definition}
\newtheorem{claim}[theorem]{Claim}

\newenvironment{poc}{\begin{proof}[Proof of claim]}{\end{proof}}

\newcommand{\eps}{\varepsilon}
\newcommand{\bN}{\mathbb{N}}
\newcommand{\bZ}{\mathbb{Z}}

\newcommand{\cB}{\mathcal{B}}
\newcommand{\cP}{\mathcal{P}}

\newcommand*{\abs}[1]{\lvert#1\rvert}

\definecolor{darkblue}{rgb}{0,0,0.5}
\definecolor{armygreen}{rgb}{0.29, 0.33, 0.13}
\definecolor{darkmagenta}{rgb}{0.55, 0.0, 0.55}
\definecolor{lightseagreen}{rgb}{0.13, 0.7, 0.67}
\definecolor{darktangerine}{rgb}{1.0, 0.66, 0.07}
\definecolor{deepcarmine}{rgb}{0.66, 0.13, 0.24}
\definecolor{darkblue}{rgb}{0.0, 0.0, 0.55}

\title{Nested cycles with no geometric crossings}

 \author{
 	Irene Gil Fern\'andez
 	\thanks{Mathematics Institute, University of Warwick, UK. Email: {\tt irene.gil-fernandez@warwick.ac.uk}.}
\and
 	Jaehoon Kim
 	\thanks{Department of Mathematical Sciences, KAIST, South Korea. Email: {\tt jaehoon.kim@kaist.ac.kr}. J.K. was supported by the POSCO Science Fellowship of POSCO TJ Park Foundation and by the KAIX Challenge program of KAIST Advanced Institute for Science-X.}
\and
 	Younjin Kim
 	\thanks{Institute of Mathematical Sciences, Ewha Womans University, South Korea. Email: {\tt younjinkim@ewha.ac.kr}. Y.K. was supported by Basic Science Research Program through the National Research Foundation of Korea(NRF) funded by the Ministry of Education (2017R1A6A3A04005963).}
\and
	Hong Liu
	\thanks{Mathematics Institute and DIMAP, University of Warwick, UK.  Email: {\tt h.liu.9@warwick.ac.uk}. H.L. was supported by the UK Research and Innovation Future Leaders Fellowship MR/S016325/1.
	}
}

\date{\today}

\begin{document}
	\maketitle
	
\begin{abstract}
	In 1975, Erd\H{o}s asked the following question: what is the smallest function~$f(n)$ for which all graphs with $n$ vertices and $f(n)$ edges contain two edge-disjoint cycles $C_1$ and $C_2$, such that the vertex set of $C_2$ is a subset of the vertex set of $C_1$ and their cyclic orderings of the vertices respect each other? We prove the optimal linear bound $f(n)=O(n)$ using sublinear expanders.
\end{abstract}
	
\section{Introduction}\label{sec-intro}
Extremal problems involving cycles have been extensively studied. In particular, what kinds of cycles can we find in graphs with large (but constant) average/minimum degree? A classical result of this sort is the Corradi-Hajnal theorem~\cite{CH63} from 1963, stating that any graph $G$ with minimum degree $\delta(G)\ge 2k$ and $|G|\ge 3k$ contains $k$ pairwise vertex-disjoint cycles. This was later extended to cycles of the same length. Egawa~\cite{Ega96}, improving an earlier result of H\"aggkvist~\cite{Hag85}, showed that large graphs with $2k$ minimum degrees contain $k$ pairwise vertex-disjoint cycles with the same length. 
By viewing cycles as minimal graphs of connectivity two or minimum degree two, these results were also generalized in the direction of finding disjoint subgraphs of certain minimum degree or connectivity in \cite{Haj, KO, Sti96, Thom}. 
See also a result of Verstra\"ete~\cite{Ver02} for vertex-disjoint cycles whose lengths form an arithmetic progression in graphs with large average degree.

In this note, we are interested in finding cycles with geometric constraints. Cycles $C_1,\dots,C_k$ in a graph $G$ are said to be \emph{nested cycles} if the vertex set of $C_{i+1}$ is a subset of the vertex set of $C_i$ for each $i\in[k-1]$. If, in addition, their edge sets are pairwise disjoint, we say they are \emph{edge-disjoint nested cycles} (see Figure~\ref{fig-nested-cyc}). In 1975, Erd\H{o}s~\cite{Erdos75} conjectured that there is a constant $c$ such that graphs with $n$ vertices and at least $cn$ edges must contain two edge-disjoint nested cycles. Bollob\'as~\cite{Bol78} proved this conjecture and asked for an extension to $k$ edge-disjoint nested cycles. This was confirmed later in 1996 by Chen, Erd\H{o}s and Staton~\cite{CES94}, who showed that $O_k(n)$ edge forces $k$ edge-disjoint nested cycles.

A stronger conjecture of Erd\H{o}s that also appeared in~\cite{Erdos75} is that there exists a constant~$C$ such that graphs with $n$ vertices and at least $Cn$ edges must contain two edge-disjoint nested cycles such that, geometrically, the edges of the inner cycle do not cross each other, in other words, if $C_1=v_1\dots v_{\ell_1}$, then $C_2$ has no two edges $v_iv_{i'}$ and $v_jv_{j'}$ with $i<j<i'<j'$. In this case, $C_1$ and~$C_2$ are said to be two \textit{nested cycles without crossings} (see Figure~\ref{fig-nested-cyc-no-cro}). The proof of the above nested cycles result of Chen, Erd\H{o}s and Staton proceeds by finding a cycle $C$ in a graph $G$ with average degree $d$ such that the average degree of the subgraph $H\subseteq G $ induced on $V(C)$ grows with $d$. One can then iterate this to get nested cycles. This argument, however, does not say anything about the shape of the cycles that can be found in $H$.

\begin{figure}[h]
	\centering
	\begin{subfigure}{0.45\linewidth}
		\centering
		\begin{tikzpicture}[scale=1.3]
		\node[inner sep= 1pt](a1) at (0,-0.1)[circle,fill]{};
		\node[inner sep= 1pt](a2) at (0.5,-0.1)[circle,fill]{};
		\node[inner sep= 1pt](a3) at (1,0.3)[circle,fill]{};
		\node[inner sep= 1pt](a4) at (1,0.8)[circle,fill]{};
		\node[inner sep= 1pt](a5) at (0.5,1.1)[circle,fill]{};
		\node[inner sep= 1pt](a6) at (0,1.1)[circle,fill]{};
		\node[inner sep= 1pt](a7) at (-0.5,0.8)[circle,fill]{};
		\node[inner sep= 1pt](a8) at (-0.5,0.3)[circle,fill]{};
		\draw[darkblue,thick] (a1)--(a2);
		\draw[darkblue,thick] (a2)--(a3);
		\draw[darkblue,thick] (a3)--(a4);
		\draw[darkblue,thick] (a4)--(a5);
		\draw[darkblue,thick] (a5)--(a6);
		\draw[darkblue,thick] (a6)--(a7);
		\draw[darkblue,thick] (a7)--(a8);
		\draw[darkblue,thick] (a8)--(a1);
		\draw[red] (a1)--(a3);
		\draw[red] (a3)--(a7);
		\draw[red] (a7)--(a2);
		\draw[red] (a2)--(a4);
		\draw[red] (a4)--(a1);
		\node[inner sep= 1pt,darkblue](a9) at (1.3,0.5){\small$C_1$};
		\node[inner sep= 1pt,red](a0) at (0.5,0.7){\small$C_2$};
	\end{tikzpicture}
\caption{Two edge-disjoint nested cycles}\label{fig-nested-cyc}
	\end{subfigure}
	\begin{subfigure}{0.45\linewidth}
	\centering
	\begin{tikzpicture}[scale=1.3]
	\node[inner sep= 1pt](a1) at (0,-0.1)[circle,fill]{};
	\node[inner sep= 1pt](a2) at (0.5,-0.1)[circle,fill]{};
	\node[inner sep= 1pt](a3) at (1,0.3)[circle,fill]{};
	\node[inner sep= 1pt](a4) at (1,0.8)[circle,fill]{};
	\node[inner sep= 1pt](a5) at (0.5,1.1)[circle,fill]{};
	\node[inner sep= 1pt](a6) at (0,1.1)[circle,fill]{};
	\node[inner sep= 1pt](a7) at (-0.5,0.8)[circle,fill]{};
	\node[inner sep= 1pt](a8) at (-0.5,0.3)[circle,fill]{};
	\draw[darkblue,thick] (a1)--(a2);
	\draw[darkblue,thick] (a2)--(a3);
	\draw[darkblue,thick] (a3)--(a4);
	\draw[darkblue,thick] (a4)--(a5);
	\draw[darkblue,thick] (a5)--(a6);
	\draw[darkblue,thick] (a6)--(a7);
	\draw[darkblue,thick] (a7)--(a8);
	\draw[darkblue,thick] (a8)--(a1);
	\draw[red] (a1)--(a3);
	\draw[red] (a3)--(a5);
	\draw[red] (a5)--(a7);
	\draw[red] (a7)--(a1);
	\node[inner sep= 1pt,darkblue](a9) at (1.3,0.5){\small$C_1$};
	\node[inner sep= 1pt,red](a0) at (0.5,0.7){\small$C_2$};
\end{tikzpicture}
\caption{Two nested cycles without crossings}\label{fig-nested-cyc-no-cro}
	\end{subfigure}
\end{figure}

Here we prove this conjecture allowing us to find nested cycles without crossings.

\begin{theorem}\label{thm-main-result}
	There exists a constant $C>0$ such that every graph $G$ with average degree at least $C$ has two nested cycles without crossings.
\end{theorem}

Our proof utilises a notion of sublinear expanders (see Section~\ref{subsec-robust-expander}), which plays an important role in some recent resolutions of long-standing conjectures, see e.g.~\cite{HKL20,KHSS17,MadLM,EH-LM}. Our embedding strategy goes in reverse order. That is, we will find the inner cycle first and then embed the outer cycle in such a way that there is no geometric crossing. This is made possible via a structure we call \emph{kraken} (see Definition~\ref{def-kraken}). The bulk of the work is to construct a kraken in a graph with large but constant average degree. 

It would be interesting to see if large constant average degree can also force $O(1)$ many nested cycles without crossings. It seems that new ideas are needed for such an extension (if it is true!).

\begin{question}
	Given $k\in\bN$, does there exist $f(k)$ such that every graph with average degree $f(k)$ contains $k$ nested cycles without geometric crossings?
\end{question}

\medskip

\noindent\textbf{Organisation}. After laying out the tools needed in Section~\ref{sec-preliminaries}, Theorem~\ref{thm-main-result} will be proved in Section~\ref{sec-main-result}, assuming that we have a kraken on our side on the battlefield. Then in Section~\ref{sec:kraken}, we show how to summon such a creature.


	
\section{Preliminaries}\label{sec-preliminaries}

For $n\in\mathbb{N}$, let $[n]:=\{1,\dots,n\}$. If we claim that a result holds for $0<a\ll b,c\ll d<1$, it means that there exist positive functions $f,g$ such that the result holds as long as $a<f(b,c)$ and $b<g(d)$ and $c<g(d)$. We will not compute these functions explicitly. In many cases, we treat large numbers as if they are integers, by omitting floors and ceilings if it does not affect the argument. We write $\log$ for the base-$e$ logarithm.

\subsection{Graphs notation}\label{subsec-graph-notation}

Given a graph $G$, denote its average degree $2e(G)/|G|$ by $d(G)$. Let $F\subseteq G$ and $H$ be graphs, and $U\subseteq V(G)$. We write $G[U]\subseteq G$ for the induced subgraph of $G$ on vertex set $U$. Denote by $G\cup H$  the graph with vertex set $V(G)\cup V(H)$ and edge set $E(G)\cup E(H)$, and write $G-U$ for the induced subgraph $G[V(G)\setminus U]$, and $G\setminus F$ for the spanning subgraph of $G$ obtained from removing the edge set of $F$. For a set of vertices $X\subseteq V(G)$ and $i\in\mathbb{N}$, the distance in $G$ between $X$ and a vertex $u$ is the length of a shortest path from $u$ to $X$;
denote
$$
N_G^i(X):=\{u\in V(G):\mbox{ the distance in $G$ between $X$ and $u$ is exactly $i$}\},
$$
and write $N_G^0(X)=X$, $N_G(X):=N_G^1(X)$, and for $i\in\mathbb{N}\cup\{0\}$, let $B_G^i(X)=\bigcup_{j=0}^iN_G^j(X)$ be the ball of radius $i$ around $X$. When the underlying graph $G$ is clear from the context, we omit the subscript $G$ and simply write $N(X), N^i(X), B^i(X)$.
For a path $P$, we write $\ell(P)$ for its length, which is the number of edges in the path. 

\subsection{Sublinear expander}\label{subsec-robust-expander}

Our proof makes use of the sublinear expander introduced by Koml\'os and Szemer\'edi \cite{KS96}. Roughly speaking, a sublinear expander is a graph in which all sets of reasonable size expand by a sublinear factor. This weak expansion and its consequence that large sets are linked by a short path drive our whole embedding argument.

We shall use the following extension by Haslegrave, Kim and Liu~\cite{HKL20}.

\begin{defin}\label{def-expander}
	Let $\varepsilon_{1}>0$ and $k\in\bN$. A graph $G$ is an $(\varepsilon_{1},k)$-\textit{expander} if for all $X\subset V(G)$ with $k/2\leq |X|\leq |G|/2$, and any subgraph $F\subseteq G$ with $e(F)\leq d(G)\cdot \eps(|X|)|X|$, we have
	$$
	|N_{G\setminus F}(X)| \geq \varepsilon(|X|)\cdot|X|,
	$$
	where
	$$
	\varepsilon(x)=\varepsilon\left(x, \varepsilon_{1}, k\right)=\left\{\begin{array}{cc}
	0 & \text { if } x<k / 5, \\
	\varepsilon_{1} / \log ^{2}(15 x / k) & \text { if } x \geq k / 5.
	\end{array}\right.
	$$ 
\end{defin}
When the parameters $\eps_1$ and $k$ are clear from the context, we will omit them and write simply $\eps(x)$.
Note that when $x\geq k/2$, $\eps(x)$ is decreasing, while $\eps(x)\cdot x$ is increasing.

Though the expansion rate of the expander above is only sublinear, the advantage of this notion is that every graph contains one such sublinear expander subgraph with almost the same average degree.

\begin{theorem}[\cite{HKL20}, Lemma 3.2]\label{thm-pass-to-expander}
	There exists some $\eps_1>0$ such that the following holds for every $k>0$. Every graph $G$ has an $(\eps_1,k)$-expander subgraph $H$ with $d(H)\geq d(G)/2$ and $\delta(H)\geq d(H)/2$.
\end{theorem}

Thanks to this theorem, by passing to an expander subgraph, it suffices to prove Theorem~\ref{thm-main-result} for expanders. A key property~\cite[Corollary 2.3]{KS96} of expanders that we will use is that there exists a short path between any two sufficiently large sets. This is formalised in the following statement.

\begin{lemma}\label{lem-short-diam}
	Let $\varepsilon_{1},k>0$. If $G$ is an $n$-vertex $(\varepsilon_{1},k)$-expander, then any two vertex sets $X_1,X_2$, each of size at least $x\geq k$, are at distance of at most $m=\frac{1}{\varepsilon_{1}}\log^3(15n/k)$ apart. This remains true even after deleting $\varepsilon(x,\eps_1,k)\cdot x/4$ vertices from $G$.
\end{lemma}

\subsection{Robust expansions}\label{subsec-exp-growth-sets}

For our proof, we need to be able to expand a set $A$ past another set $U$ as long as $U$ does not interfere with each sphere around $A$ too much. The following notion makes this precise.

\begin{defin}
	For $\lambda>0$ and $k\in\mathbb{N}$, we say that a vertex set $U$ in a graph $G$ is~\emph{$(\lambda,k)$-thin around $A$} if $U\cap A=\varnothing$ and, for each $i\in\mathbb{N}$,
	$$
	|N_G(B_{G-U}^{i-1}(A))\cap U|\leq \lambda i^k.
	$$
\end{defin}

As shown in the lemma below, the rate of expansion for every small set is almost exponential in an expander even after deleting a thin set around it. Its proof follows essentially~\cite[Proposition 3.5]{HKL20}.

\begin{prop}\label{prop-exp-HL}
	Let $0<1/d\ll \varepsilon_1\ll 1/\lambda, 1/k$ and $1\leq r\leq\log n$. Suppose $G$ is an $n$-vertex $(\varepsilon_1,\varepsilon_1d)$-expander with $\delta(G)\ge d$, and $X,Y$ are sets of vertices with $|X|\ge 1$ and $|Y|\leq \frac{1}{4}\varepsilon(|X|)\cdot |X|$. Let $W$ be a $(\lambda,k)$-thin set around $X$ in $G-Y$. Then, for each $1\leq r\leq \log n$, we have
	$$|B^r_{G-W-Y}(X)|\geq\exp(r^{1/4}).$$
\end{prop}
\begin{proof}
	For each $i\ge 0$, let $Z_i= B^{i}_{G-W-Y}(X)$.  As $W$ is $(\lambda,k)$-thin around $X$ in $G-Y$, we obtain that for each $i\ge 0$,
	\begin{align}
		\abs{Z_{i+1}} &= \abs{Z_i} + \abs{N_G(Z_i)\setminus(Y\cup W)}\ge \abs{Z_i} + \eps(\abs{Z_i})\abs{Z_i} - \abs{Y}-\abs{N_{G-Y}(Z_i)\cap W}\nonumber \\
		&\geq \abs{Z_i} + \frac{3}{4}\eps(\abs{Z_i})\abs{Z_i} - \lambda i^k.\label{this}
	\end{align}
	Indeed, this follows from the expansion of $G$ if $|Z_i|\ge \eps_1d$; and if $|Z_i|\le \eps_1d$, it follows from $\delta(G)\ge d$. Thus we have $\abs{Z_0}=|X|$ and $|Z_1|\ge d/2\ge \eps_1d$.

	Let $x=|Z_1|\ge \eps_1d$ and define $f(z)=\exp(z^{1/4})$ and $g(z):= x + \frac{1}{2}\eps(x)x\cdot (z-1)$.
	We first use induction on $i$ to show that for each $1\le i\le \log^4x$, $\abs{Z_i}\ge g(i)$. The base case $i=1$ is trivial.
	
	Now for $1\leq i\leq\log^4x$ we use \eqref{this} together with the induction hypothesis and the facts that $\eps(z)z$ is increasing when $z\ge x$ and $\frac{1}{4}\eps(x)x > \lambda i^k$ to obtain 
	\[\abs{Z_{i+1}}\geq \abs{Z_i}+ \frac{1}{2}\eps(x)x = \abs{Z_i}+ g(i+1)-g(i)\ge g(i+1).\]
	
	We may then assume $i>\log^4x$, as $f(i)\le g(i)\le \abs{Z_i}$ when $i\le \log^4x$. Now, as $i>\log^4x$ is sufficiently large (due to $x\ge \eps_1d$), we get $\lambda i^k< \frac{ f(i)}{i^{3/4}}$.
	Finally, noting that $f(i+1)-f(i)\le \frac{f(i)}{i^{3/4}}$ and $\eps(\abs{Z_i})\abs{Z_i}\ge \eps(f(i))f(i)\ge\frac{\eps_1f(i)}{i^{1/2}}$, we have
	\begin{align*}
		\abs{Z_{i+1}}&\ge \abs{Z_i}+\frac{3}{4}\eps(\abs{Z_i})\abs{Z_i} - \lambda i^k \geq 
		\abs{Z_i}+\frac{3\eps_1f(i)}{4i^{1/2}}- \frac{ f(i)}{i^{3/4}}\\
		&\geq \abs{Z_i}+\frac{ f(i)}{i^{3/4}} \geq \abs{Z_i}+ f(i+1)-f(i)\ge f(i+1),
	\end{align*}
	as desired.
\end{proof}

We will use the following lemma \cite[Lemma 3.12]{EH-LM} to find a linear size vertex set with polylogarithmic diameter in $G$ while avoiding an arbitrary set of size $o(n/\log^2n)$. 

\begin{lemma}\label{lem-find-large-ball}
	Let $0<1/d\ll\eps_1<1$ and let $G$ be an $n$-vertex $(\eps_1,\eps_1d)$-robust-expander with $\delta(G)\geq d$. For any $W\subseteq V(G)$ with $|W|\leq \eps_1n/100\log^2n$, there is a set $B\subseteq G-W$ with size at least $n/25$ and diameter at most $100\eps_1^{-1}\log^3n$.
\end{lemma}

\section{Proof of Theorem \ref{thm-main-result}}\label{sec-main-result}
A key structure we use in our proof is a kraken defined below.
The idea of the proof is depicted in Figure~\ref{fig:Kraken}. We embed the desired nested cycles by taking the cycle in a kraken to be the inner cycle and linking the arms of the kraken iteratively to get the outer cycle so that the cyclic orderings of the vertices of both cycles respect each other.


\begin{defin}\label{def-kraken}
	For $k,m,s\in\mathbb{N}$, a graph $K$ is a \emph{$(k,m,s)$-kraken} if it contains a cycle $C$ with vertices $v_1,\dots,v_k$ (in the order of the cycle $C$), vertices $u_{i,1},u_{i,2}\in V(G)\setminus V(C), i\in [k]$, and subgraphs $A=\bigcup_{i=1}^{k}\big(A_{i,1}\cup A_{i,2}\big)$ and $R=\bigcup_{i=1}^{k}\big(R_{i,1}\cup R_{i,2}\big)$, where
	\begin{itemize}\itemsep=0pt
		\item $\{A_{i,j}:i\in[k],\, j\in[2]\}$ is a collection of pairwise disjoint sets of size $s$ lying in $V(G)\setminus V(C)$ with $u_{i,j}\in A_{i,j}$, each with diameter at most $m$.
		\item $\{R_{i,j}:i\in[k],\, j\in[2]\}$ is a collection of pairwise internally vertex disjoint paths such that $R_{i,j}$ is a path between $v_i$ and $u_{i,j}$ of length at most $10m$ with internal vertices disjoint from $V(C)\cup(V(A)\setminus V(A_{i,j}))$. 
	\end{itemize}
\end{defin}

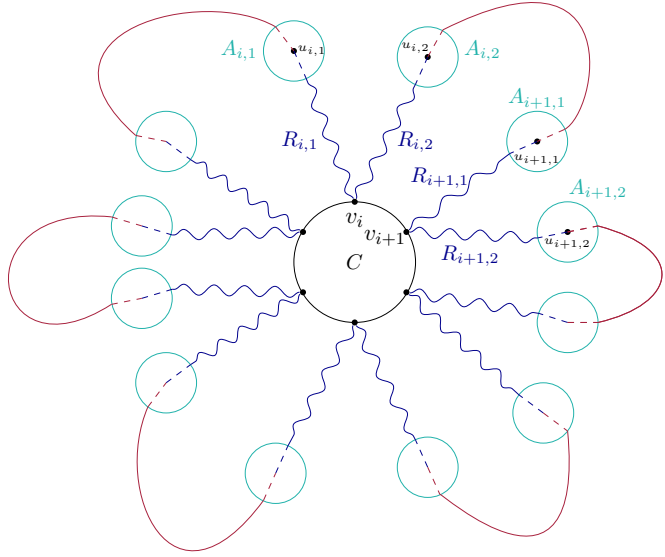
\begin{figure}[h]
	\centering
	\scalebox{0.8}{	
		\begin{tikzpicture}	
	\draw[black] (0,0) circle (1cm);
	\node[inner sep= 1pt](a) at (0,0){\small$C$};
	\node[inner sep= 1pt](a1) at (0.85,0.5)[circle,fill]{};
	\node[inner sep= 1pt](b) at (1,2){\small\textcolor{darkblue}{$R_{i,2}$}};
	\node[inner sep= 1pt](b) at (1.4,1.4){\small\textcolor{darkblue}{$R_{i+1,1}$}};
	\node[inner sep= 1pt](b) at (1.9,0.1){\small\textcolor{darkblue}{$R_{i+1,2}$}};
	\node[inner sep= 1pt](b) at (3,2.7){\small\textcolor{lightseagreen}{$A_{i+1,1}$}};
	\node[inner sep= 1pt](b) at (4,1.2){\small\textcolor{lightseagreen}{$A_{i+1,2}$}};
	\node[inner sep= 1pt](b) at (-0.9,2){\small\textcolor{darkblue}{$R_{i,1}$}};
	\node[inner sep= 1pt](b) at (0,0.7){$v_i$};
	\node[inner sep= 1pt](b) at (0.5,0.4){$v_{i+1}$};
	\node[inner sep= 1pt](b) at (2.1,3.5){\small\textcolor{lightseagreen}{$A_{i,2}$}};
	\node[inner sep= 1pt](b) at (-1.9,3.5){\small\textcolor{lightseagreen}{$A_{i,1}$}};
	\node[inner sep= 1pt](b) at (1,3.55){\tiny$u_{i,2}$};
	\node[inner sep= 1pt](a2) at (3.5,0.5)[circle,fill]{};
	\node[inner sep= 1pt](b) at (3,1.7){\tiny$u_{i+1,1}$};
	\node[inner sep= 1pt](b) at (3.5,0.3){\tiny$u_{i+1,2}$};
	\node[inner sep= 1pt](a2) at (3,2)[circle,fill]{};
	\node[inner sep= 1pt](a2) at (0,1)[circle,fill]{};
	\node[inner sep= 1pt](a3) at (-0.85,0.5)[circle,fill]{};
	\node[inner sep= 1pt](a4) at (-0.85,-0.5)[circle,fill]{};
	\node[inner sep= 1pt](a5) at (0,-1)[circle,fill]{};
	\node[inner sep= 1pt](a6) at (0.85,-0.5)[circle,fill]{};
	\draw[lightseagreen] (3.5,0.5) circle (0.5cm);
	\draw[lightseagreen] (3,2) circle (0.5cm);
	\draw[decorate, decoration=snake, segment length=5mm,darkblue] (a1) -- (3,0.4);
	\draw[decorate, decoration=snake, segment length=6mm,darkblue] (a1) -- (2.55,1.8);
	\draw[darkblue,dashed]  (3.5,0.5) -- (3,0.4);
	\draw[darkblue,dashed]  (3,2) -- (2.55,1.8);
	\draw[deepcarmine] (4,0.6) .. controls (5.5,0.2) and (5.3,-0.7) .. (4,-1);
	\draw[deepcarmine,dashed]  (3.5,0.5) -- (4,0.6);
	\draw[deepcarmine,dashed] (3.5,-1) -- (4,-1);
	\draw[lightseagreen] (1.2,3.4) circle (0.5cm);
	\node[inner sep= 1pt](aa) at (1.2,3.4)[circle,fill]{};
	\node[inner sep= 1pt](aaa) at (-0.7,3.5){\tiny$u_{i,1}$};
	\draw[lightseagreen] (-1,3.5) circle (0.5cm);
	\node[inner sep= 1pt](aa) at (-1,3.5)[circle,fill]{};
	\draw[decorate, decoration=snake, segment length=5mm,darkblue] (a2) -- (1,2.95);
	\draw[decorate, decoration=snake, segment length=5mm,darkblue] (a2) -- (-0.8,3.05);
	\draw[darkblue,dashed]  (1.2,3.4) -- (1,2.95);
	\draw[darkblue,dashed]  (-1,3.5) -- (-0.8,3.05);
	\draw[lightseagreen] (-3.5,0.6) circle (0.5cm);
	\draw[lightseagreen] (-3.1,2) circle (0.5cm);
	\draw[decorate, decoration=snake, segment length=5mm,darkblue] (a3) -- (-3,0.45);
	\draw[decorate, decoration=snake, segment length=4mm,darkblue] (a3) -- (-2.7,1.7);
	\draw[darkblue,dashed]  (-3.5,0.6) -- (-3,0.45);
	\draw[darkblue,dashed]  (-3.1,2) -- (-2.7,1.7);
	\draw[lightseagreen] (-3.5,-0.6) circle (0.5cm);
	\draw[lightseagreen] (-3.1,-2) circle (0.5cm);
	\draw[decorate, decoration=snake, segment length=5mm,darkblue] (a4) -- (-3,-0.45);
	\draw[decorate, decoration=snake, segment length=4mm,darkblue] (a4) -- (-2.7,-1.7);
	\draw[darkblue,dashed]  (-3.5,-0.6) -- (-3,-0.45);
	\draw[darkblue,dashed]  (-3.1,-2) -- (-2.7,-1.7);
	\draw[lightseagreen] (1.2,-3.4) circle (0.5cm);
	\draw[lightseagreen] (-1.3,-3.5) circle (0.5cm);
	\draw[decorate, decoration=snake, segment length=5mm,darkblue] (a5) -- (1,-2.95);
	\draw[decorate, decoration=snake, segment length=5mm,darkblue] (a5) -- (-1.1,-3.05);
	\draw[darkblue,dashed]  (1.2,-3.4) -- (1,-2.95);
	\draw[darkblue,dashed]  (-1.3,-3.5) -- (-1.1,-3.05);
	\draw[lightseagreen] (3.5,-1) circle (0.5cm);
	\draw[lightseagreen] (3.1,-2.5) circle (0.5cm);
	\draw[decorate, decoration=snake, segment length=5mm,darkblue] (a6) -- (3.03,-0.85);
	\draw[decorate, decoration=snake, segment length=4.5mm,darkblue] (a6) -- (2.7,-2.2);
	\draw[darkblue,dashed]  (3.5,-1) -- (3.03,-0.85);
	\draw[darkblue,dashed]  (3.1,-2.5) -- (2.7,-2.2);
	\draw[deepcarmine] (4,0.6) .. controls (5.5,0.2) and (5.3,-0.7) .. (4,-1);
	\draw[deepcarmine,dashed]  (3.5,0.5) -- (4,0.6);
	\draw[deepcarmine,dashed] (3.5,-1) -- (4,-1);
	\draw[deepcarmine] (3.5,-2.8) .. controls (3.8,-6) and (1.4,-3.8) .. (1.4,-3.85);
	\draw[deepcarmine,dashed]  (3.1,-2.5) -- (3.5,-2.8);
	\draw[deepcarmine,dashed] (1.2,-3.4) -- (1.4,-3.85);
	\draw[deepcarmine] (-1.5,-3.95) .. controls (-3,-6) and (-4,-3.8) .. (-3.4,-2.4);
	\draw[deepcarmine,dashed]  (-1.5,-3.95) -- (-1.3,-3.5);
	\draw[deepcarmine,dashed] (-3.4,-2.4) -- (-3.1,-2);
	\draw[deepcarmine] (-4,-0.7) .. controls (-6,-2) and (-6.5,1) .. (-4,0.75);
	\draw[deepcarmine,dashed]  (-4,0.75) -- (-3.5,0.6);
	\draw[deepcarmine,dashed] (-4,-0.7) -- (-3.5,-0.6);
	\draw[deepcarmine] (-3.6,2.1) .. controls (-5.5,3.8) and (-2.5,4.6) .. (-1.3,3.9);
	\draw[deepcarmine,dashed]  (-3.6,2.1) -- (-3.1,2);
	\draw[deepcarmine,dashed] (-1.3,3.9) -- (-1,3.5);
	\draw[deepcarmine] (1.45,3.85) .. controls (4.5,5.4) and (4,2.5) .. (3.45,2.2);
	\draw[deepcarmine,dashed]  (1.2,3.4) -- (1.45,3.85);
	\draw[deepcarmine,dashed] (3.45,2.2) -- (3,2);
	\end{tikzpicture}
}
		\caption{A $(6,m,s)$-kraken with an inner cycle $C$ of length $6$. Each vertex $v_i$ in $C$ has two arms attached to it. The arms of the kraken and the (red) paths linking them form the outer cycle.\label{fig:Kraken}}
\end{figure}

We usually write a kraken as a tuple $(C, A_{i,j}, R_{i,j}, u_{i,j})$, $i\in[k], j\in[2]$. The following lemma finds a kraken in any expander with average degree at least some large constant.

\begin{lemma}\label{lem-main}
Let $0<1/d\ll \eps_1<1$ and let $G$ be an $n$-vertex $(\eps_1,\eps_1d)$-expander with $\delta(G)\geq d$. Let $L$ be the set of vertices in $G$ with degree at least $\log^{100}n$ and let $m=100\eps_1^{-1}\log^3n$. Then, there exists a $(k,m,\log^{10}n)$-kraken $(C, A_{i,j}, R_{i,j}, u_{i,j})$, $i\in[k], j\in[2]$, in $G$ for some $k\le \log n$ such that
\begin{itemize}
	\item either $\{u_{i,j}:\.i\in[k], j\in[2]\}\subseteq L$;
	
	\item or $|L|\le 2\log n$ and any distinct $u_{i,j}, u_{i',j'}\not\in L$ are a distance at least $\sqrt{\log n}$ apart in $G-L$.
\end{itemize}
\end{lemma}

\begin{proof}[Proof of Theorem~\ref{thm-main-result}]
	Let $L$ be the set of high degree vertices and $K=(C, A_{i,j}, R_{i,j}, u_{i,j})$, $i\in[k],j\in [2]$ be the kraken as in Lemma \ref{lem-main}. We will embed, for each $i\in \bZ_k$, a path $P_i$ between $u_{i,2}$ and $u_{i+1,1}$  of length at most $30m$, such that all paths $P_i$ are internally pairwise disjoint and also disjoint from $C$ and $R_{i,j}$, $i\in[k], j\in[2]$. Here, indices are taken modulo $k$.
	Such paths $P_i$, $i\in[k]$, together with $C$ and $R_{i,j}$, $i\in[k], j\in[2]$, form the desired nested cycles without crossings. 
	
	If the first alternative in Lemma~\ref{lem-main} occurs, i.e. all $u_{i,j}\in L$, $i\in[k], j\in[2]$, then we can iteratively find the desired paths $P_i$, $i\in[k]$, by linking $N(u_{i,2})$ and $N(u_{i+1,1})$ avoiding previously built paths and the kraken $K$, using Lemma~\ref{lem-short-diam}. Indeed, the number of vertices to avoid is at most $|V(K)|+k\cdot 30m\le \log^{20}n$, which is much smaller than the degree of vertices in $L$.

	We may then assume that $|L|\leq 2\log n$ and distinct $u_{i,j}, u_{i',j'}\not\in L$ are at distance at least $\sqrt{\log n}$ apart in $G'=G-L$. Let $V'\subseteq V(C)$ be the set of vertices not linked to vertices in $L$, i.e.~$V'=\{v_i\in V(C):~\{u_{i,1}, u_{i,2}\}\not\subseteq L\}$.
	
	For each $v_i\in V'$ and $j\in[2]$ with $u_{i,j}\not\in L$, write $Y_{i,j}=(\cup_{i'\in[k], j'\in[2]}R_{i',j'}\setminus\{u_{i,j}\}) \cup V(C)$. Note that $|Y_{i,j}|\le \log^5n$. Recall that $|A_{i,j}|=\log^{10}n$. Applying Proposition~\ref{prop-exp-HL} with $(X,Y,W)_{\ref{prop-exp-HL}}=(A_{i,j}, Y_{i,j}\cup L ,\varnothing)$, we can expand $A_{i,j}$ in $G'$ avoiding $Y_{i,j}$ to get $A_{i,j}^*:=B_{G'-Y_{i,j}}^r(A_{i,j})$ of size at least $\log^{30}n$, where $r=(\log\log n)^{10}$. Moreover, note that for distinct $v_i,v_{i'}\in V'$ and $j,j'\in[2]$ with $u_{i,j},u_{i',j'}\not\in L$, $u_{i,j}$ and $u_{i',j'}$ are at distance at least $\sqrt{\log n}$ apart in $G'$, therefore the corresponding $A_{i,j}^*$ and $A_{i',j'}^*$ are disjoint. 
	
	Finally, for $v_i\in V'$ and $j\in[2]$ with $u_{i,j}\in L$ and all $v_i\in V(C)\setminus V'$ and $j\in[2]$ whose corresponding $u_{i,j}$ lie in $L$, we can choose pairwise disjoint $A_{i,j}^*\subseteq N(u_{i,j})\setminus (\cup_{v_{i'}\in V', j'\in[2]}A^*_{i',j'})$, each of size $\log^{30}n$. For each $i\in\bZ_k$,  link $A_{i,2}^*$ and $A_{i+1,1}^*$ in $G$ to get a path $Q_i$ with length at most $m$ using Lemma~\ref{lem-short-diam}, avoiding previously built path $Q_j$ and $K$ (indices are taken modulo $k$). The desired path $P_i$ between $u_{i,2}$ and $u_{i+1,1}$ can be obtained by extending $Q_i$ in $A_{i,2}^*\cup A_{i+1,1}^*$.
	
	This concludes the proof.
\end{proof}

\section{Release the Kraken!}\label{sec:kraken}
In this section, we prove Lemma~\ref{lem-main}. The idea of the proof is the following. We take a shortest cycle $C$ in the graph, and consider two cases depending on whether the set $L$ of high degree vertices is large or not. For the case where this set $L$ is large, we want to link each vertex on the cycle $C$ to two different vertices in $L$ by expanding vertices in $C$ using Proposition \ref{prop-exp-HL}. If there is a small number of high degree vertices, we will use the fact that the graph $G-L$ has relatively small maximum degree, and so we can find within $G-L$ many large connected set of vertices that are pairwise far apart. Then we will expand vertices in $C$ to link them to these connected sets. 

One difficulty here is that the number of paths we need to build to link either high degree vertices or large connected sets to vertices in $C$ could be as large as $\Omega(\log n)$; while the degree of each vertex in $C$ could be as small as $O(1)$. We have to be careful when embedding these paths so that no vertex in $C$ gets isolated.

Let us first see how we can link vertices on a shortest cycle to high degree vertices in $L$ as follows.
\begin{lemma}\label{lem:link-L}
	Let $0<1/d\ll \eps_1<1$ and let $G$ be an $n$-vertex $(\eps_1,\eps_1d)$-expander with $\delta(G)\geq d$. Let $L$ be the set of vertices in $G$ with degree at least $\log^{100}n$ and let $m=100\eps_1^{-1}\log^3n$. Let $C$ be a shortest cycle in $G$ with vertices $v_1,\ldots, v_k$ and let $\mathcal{P}$ be a maximal collection of paths in $G$ from $V(C)$ to $L$ such that:
	\begin{itemize}\itemsep=0pt
		\item each $v\in V(C)$ is linked to at most $2$ vertices in $L$;
		\item all paths are pairwise disjoint outside of $V(C)$ with internal vertices in $V(G)\setminus(V(C)\cup L)$;
		\item each path has length at most $10m$.
	\end{itemize}
	Subject to $|\mathcal{P}|$ being maximal, let $\ell(\cP):=\sum_{P\in\mathcal{P}}\ell(P)$ be minimised. 
	
	Let $U=V(C)\cup V(\mathcal{P})$. Then, 
	\begin{itemize}
		\item[(i)] for any vertex $v\in C$ that is in less than 2 paths in $\cP$, the set $U\setminus\{v\}$ is $(10,2)$-thin around $v$ in $G$; and
		
		\item[(ii)] at least $\min\{|L|, 2k\}$ many vertices in $L$ are linked to paths in $\cP$.
	\end{itemize}
\end{lemma}

\begin{proof}
	First of all, $|C|=k\leq 2\log_d n\leq\log n$ due to $\delta(G)\ge d$. Consequently, as there are at most two paths containing each vertex in $C$, we have that $|\mathcal{P}|\le 2|C|\le 2\log n$. So
	\begin{equation}\label{eq-size-U-new}
		|U|\leq \log n+2\log n\cdot 10m\leq 4000\eps_1^{-1}\cdot\log^4 n.
	\end{equation}
	
	Notice that as $C$ is a shortest cycle in $G$, for any $i\in\mathbb{N}$, we have
	\begin{equation}\label{eq-lim-contact-with-C}
		|N_G(B_{G-V(C)}^{i-1}(v))\cap V(C)|\leq 2i,
	\end{equation}
	for otherwise, replacing the segment of $C$ intersecting $N_G(B_{G-V(C)}^{i-1}(v))$ with a path of length at most $2i$ in $B_{G-V(C)}^{i-1}(v)\cup N_G(B_{G-V(C)}^{i-1}(v))$ results in a shorter cycle than $C$, contradicting the minimality of $C$.
	
	Let us first prove (i). Suppose $v\in C$ is in less than 2 paths in $\cP$.

	\begin{claim}\label{claim:U-thin}
		$U\setminus\{v\}$ is $(10,2)$-thin around $v$ in $G$.
	\end{claim}	
	\begin{poc}
		Fix first an arbitrary path $P$ in $\mathcal{P}$. Similar to~\eqref{eq-lim-contact-with-C}, by the minimality of $\ell(\cP)$, we have for any $i\in\mathbb{N}$ that
		\begin{equation}\label{eq-lim-contact-with-P}
			|N_G(B_{G-U\setminus\{v\}}^{i-1}(v))\cap V(P)|\leq i.
		\end{equation}
		
		Next, fix $i\in\bN$ and consider the set $B_C^{4i}(v)$. Let $\widehat{P}$ be a path in $\mathcal{P}$ whose endvertex $\hat{v}$ in $C$ is not in $B_C^{4i}(v)$ (if such a path exists). Notice that
		\begin{equation}\label{eq:far-path}
			N_G(B_{G-U\setminus\{v\}}^{i-1}(v))\cap \widehat{P}=\varnothing.
		\end{equation} 
		To see this, split $\widehat{P}$ into two parts $\widehat{P}^0$ and $\widehat{P}^1$, with $\widehat{P}^0$ consisting of the nearest $2i$ vertices to $C$ in $\widehat{P}$, and $\widehat{P}^1$ being the remaining ones. If there is a vertex $y\in N_G(B_{G-U\setminus\{v\}}^{i-1}(v))\cap \widehat{P}^0$, let $\widehat{Q}$ be the path between $v$ and $y$ of length $i$ with internal vertices in $B_{G-U\setminus\{v\}}^{i-1}(v)$. Then $\widehat{Q}\cup \widehat{P}^0$ contains a path $Q$ between $v$ and $\hat{v}$ of length at most $3i$ with $V(Q)\cap V(C)=\{v,\hat{v}\}$. Replacing the segment between $v$ and $\hat{v}$ in $C$ (which is of length at least $4i$) with $Q$ results in a shorter cycle than $C$, a contradiction. So, we must have $N_G(B_{G-U\setminus\{v\}}^{i-1}(v))\cap \widehat{P}^0=\varnothing$. If there is $y'\in N_G(B_{G-U\setminus\{v\}}^{i-1}(v))\cap \widehat{P}^1$, then $y'$ is at distance at most $i$ from $v$ in $G-U\setminus\{v\}$ and at distance at least $2i$ from $\hat{v}$ on $\widehat{P}$, so $B_{G-U\setminus\{v\}}^{i-1}(v)\cup \widehat{P}^1$ contains a shorter path between $V(C)$ and $L$, contradicting the minimality of $\ell(\cP)$. Hence, $N_G(B_{G-U\setminus\{v\}}^{i-1}(v))\cap \widehat{P}^1=\varnothing $, as desired.
		
		Let $\mathcal{P}_{4i}$ be the set of paths in $\mathcal{P}$ whose endvertices in $C$ are in $B_C^{4i}(v)$. By~\eqref{eq-lim-contact-with-P} and~\eqref{eq:far-path}, we see that
		\begin{equation*}
			\big|\big(N_G(B_{G-U\setminus\{v\}}^{i-1}(v))\big)\cap V(\mathcal{P})\big|=\Big|\bigcup_{P\in\mathcal{P}_{4i}}\Big(N_G(B_{G-U\setminus\{v\}}^{i-1}(v))\cap V(P)\Big)\Big|\leq 2\cdot 4i\cdot i= 8 i^2.
		\end{equation*}
		This, together with~\eqref{eq-lim-contact-with-C}, implies that
		\begin{equation*}
			|N_G(B_{G-U\setminus\{v\}}^{i-1}(v))\cap (U\setminus\{v\})\}|\leq |N_G(B_{G-U\setminus\{v\}}^{i-1}(v))\cap V(C)|+|N_G(B_{G-U\setminus\{v\}}^{i-1}(v))\cap V(\mathcal{P})|\leq 10 i^2.
		\end{equation*}
		Therefore, $U\setminus\{v\}$ is $(10,2)$-thin around $v$ as claimed. 
	\end{poc}

	Let us now turn to part (ii). Suppose for a contradiction that less than $\min\{|L|, 2k\}$ many vertices in $L$ are linked to paths in $\cP$. Then either $|\cP|<|L|<2k$ or $|L|\ge 2k$ and $|\cP|\le 2k-1$. In both cases, there are a vertex $v\in V(C)$ which is in less than $2$ paths in $\mathcal{P}$ and a high degree vertex $w\in L\setminus V(\mathcal{P})$ not linked to any path in $\cP$. Then, by Claim~\ref{claim:U-thin}, $U\setminus\{v\}$ is $(10,2)$-thin around $v$ in $G$.

	Setting $r=(\log\log n)^{10}$ and applying Proposition \ref{prop-exp-HL} with $(X,Y,W)_{\ref{prop-exp-HL}}=(\{v\},\varnothing, U\setminus\{v\})$ gives
	\begin{equation}\label{eq-use-lem-exp-case1}
		|B^r_{G-U\setminus\{v\}}(v)|\geq \exp(r^{\frac{1}{4}})=\log^{100}n.
	\end{equation}
	Recall that we have a vertex $w\in L\setminus V(\mathcal{P})$ not in any paths in $\cP$ with degree $\deg(w)\geq \log^{100}n$. Thus, by Lemma~\ref{lem-short-diam},~\eqref{eq-size-U-new} and~\eqref{eq-use-lem-exp-case1}, we can connect~$B^r_{G-U\setminus\{v\}}(v)$ and $N_G(w)$ with a path of length at most $m$ in $G-U\setminus\{v\}$, which extends in $B^{r}_{G-U\setminus\{v\}}(v)\cup\{w\}$ to a path between $v$ and $w$ in $G-U\setminus\{v\}$ with length at most $10m$, contradicting the maximality of $\mathcal{P}$. Thus, at least $\min\{|L|, 2k\}$ many vertices in $L$ are linked to paths in $\cP$ as desired.
\end{proof}

We are now ready to prove Lemma~\ref{lem-main}.

\begin{proof}[Proof of Lemma \ref{lem-main}]
	Let $C$ and $\cP$ be as in Lemma~\ref{lem:link-L}. We distinguish two cases depending on how many high degree vertices there are in $L$.
	
	\medskip
	
	\noindent\textsf{Case 1:} Suppose $|L|\geq 2k$. Then by Lemma~\ref{lem:link-L}~(ii), at least $2k$ vertices in $L$ are linked in paths in $\cP$. By the choice of $\cP$, we see that each vertex in $V(C)$ is in exactly $2$ paths in $\mathcal{P}$. Label the paths $R_{i,j}$ so that $v_i$ is the endvertex of $R_{i,j}$ in $C$ for $i\in[k], j\in[2]$, and let $u_{i,j}$ be the endvertex of $R_{i,j}$ in $L$. As vertices in $L$ have high degree, we can comfortably choose pairwise disjoint $A_{i,j}$ from $N(u_{i,j})$, each of size $\log^{10}n$, yielding the first alternative.

\medskip

\noindent\textsf{Case 2:} Suppose $|L|<2k$. Then by Lemma~\ref{lem:link-L}~(ii) again, we see that every vertex in $L$ is linked to a path in $\cP$, i.e.~$L\subseteq V(\cP)$. Note that $|V(\cP)|\le 2k\cdot 10m\le \log^5n$.
Relabelling if necessary, let $k'\leq k$ be such that $v_1,\dots,v_{k'}$ are the vertices in $C$ that are not linked in $\cP$ to two vertices in $L$. Let 
$$V'=\{v_1,\dots,v_{k'}\} \quad \text{ and } G':= G-(V(\cP)\setminus V')\subseteq  G-L.$$ Then, by the definition of $L$, $\Delta(G')\leq\log^{100}n$. 

\begin{claim}\label{cl:far-apart-balls}
	There are sets $B_i\subseteq V(G')$, $i\in[n^{1/8}]$, each of diameter at most $m$ and size $n^{1/8}$, that are at distance at least $4\sqrt{\log n}$ from each other and from $V(C)$ in $G'$.
\end{claim}
\begin{poc}
	Take a maximal collection of sets $B_i$, $i\in[s]$, with the claimed size, which are pairwise far apart and far from $V(C)$ in $G'$. If $s<n^{1/8}$, then
	$$\big|B_{G'}^{4\sqrt{\log n}}(\cup_{i\in[s]}B_i\cup V(C))\big|\leq 2\cdot (s n^{1/8}+\log n)\cdot \Delta(G')^{4\sqrt{\log n}} <\sqrt{n}.$$
    Then by Lemma~\ref{lem-find-large-ball}	with $W_{\ref{lem-find-large-ball}}=V(\cP)\cup B_{G'}^{4\sqrt{\log n}}(\cup_{i\in[s]}B_i\cup V(C))$, we can find another large set with small diameter far apart from $\cup_{i\in[s]}B_i\cup V(C)$ in $G'$, a contradiction.
\end{poc}

Let $B_i, i\in[n^{1/8}]$, be the  sets from the above claim. For each $i\in[n^{1/8}]$, let $B_i'\subseteq B_i$ be a connected subset of size $\log^{10}n$, and set $\cB=\{B_i'\}_{i\in[n^{1/8}]}$. Let $\mathcal{P}'$ be a maximal collection of paths in $G'$ from $V'$ to $V(\cB)$ such that
\begin{itemize}\itemsep=0pt
	\item each $v\in V'$ is in at most two paths of $\mathcal{P}'$ and each set in $\cB$ is linked to at most one path;
	\item all paths are pairwise disjoint outside of $V'$ with internal vertices in $V(G')\setminus V(\cB)$;
	\item the length of each path is at most $10m$.
\end{itemize}
Subject to $|\mathcal{P'}|$ being maximal, let $\ell(\cP'):=\sum_{P\in\mathcal{P'}}\ell(P)$ be minimised.

Suppose, for contradiction, that there is some $v\in V'$ which is in less than two paths in $\mathcal{P}'$. Let $\cB'\subseteq \cB$ be the collection of sets linked to some path in $\cP'$, then $|V(\cB')|\le 2|V'|\cdot \log^{10}n\le \log^{12} n$. Let $U'':=V(C)\cup V(\cP)\setminus \{v\}$ and let $U':=V(C)\cup V(\cP)\cup V(\cP')\setminus \{v\}$, so $|U'|\le \log n+2k\cdot 10m\le \log^5n$. Note that $|U'|+|V(\cB')|<|\cB|/2$, thus taking the subfamily $\cB''\subseteq \cB$ of connected sets that are disjoint from $U'\cup V(\cB')$, we have $|\cB''|\ge |\cB|/2$. 

By the choice of $V'$, the vertex $v\in V'$ is in less than 2 paths in $\cP$. Thus, by Lemma~\ref{lem:link-L}~(i), $U''$ is $(10,2)$-thin around $v$ in $G$. On the other hand, by the minimality of $C$ and $\ell(\cP')$, as in~\eqref{eq-lim-contact-with-P} and~\eqref{eq:far-path}, we have for any $i\in\bN$, for any path $P\in\cP'$, and for each path $\tilde{P}\in\cP'$ whose endvertex $\tilde{v}$ in $C$ is not in $B_C^{4i}(v)$ (if such a path exists) that
$$|N_G(B_{G-U'}^{i-1}(v))\cap V(P)|\leq i \quad\text{ and }\quad	N_G(B_{G-U'}^{i-1}(v))\cap \tilde{P}=\varnothing.$$
Thus, writing $\mathcal{P}'_{4i}$ for the set of paths in $\mathcal{P}'$ whose endvertices in $C$ are in $B_C^{4i}(v)$, we get that
\begin{equation*}
	\big|\big(N_G(B_{G-U'}^{i-1}(v))\big)\cap V(\mathcal{P}')\big|=\Big|\bigcup_{P\in\mathcal{P}'_{4i}}\Big(N_G(B_{G-U'}^{i-1}(v))\cap V(P)\Big)\Big|\leq 2\cdot 4i\cdot i= 8 i^2.
\end{equation*}
This, together with the fact that $U''$ is $(10,2)$-thin around $v$ in $G$, implies that
\begin{align*}
	\big|\big(N_G(B_{G-U'}^{i-1}(v))\big)\cap U'\big| &\le 	\big|\big(N_G(B_{G-U'}^{i-1}(v))\big)\cap U''\big| +	\big|\big(N_G(B_{G-U'}^{i-1}(v))\big)\cap V(\cP')\big| \\
	&\leq 	\big|\big(N_G(B_{G-U''}^{i-1}(v))\big)\cap U''\big| +8i^2\leq 18 i^2,
\end{align*}
where the second inequality holds as $B_{G-U'}^{i-1}(v)\subseteq B_{G-U''}^{i-1}(v)\subseteq V(G)\setminus U''$. 

Thus,
$U'$ is $(18,2)$-thin around $v$ in $G$. Then as in~\eqref{eq-use-lem-exp-case1}, we can expand $v$ in $G-U'\subseteq G'$. Furthermore, by Claim~\ref{cl:far-apart-balls}, $B_{G-U'}^r(v)$ is disjoint from $V(\cB)$. As $U'\cup V(\cB')$ is much smaller than $V(\cB'')$ and $B_{G-U'}^r(v)$, we can find a path of length at most $m$ in $G-U'-V(\cB')$ between $V(\cB'')$ and $B_{G-U'}^r(v)$ using Lemma~\ref{lem-short-diam}. Extend this path to $v$ yields one more path between $V'$ and $V(\cB)$, contradicting the maximality of $\cP'$.

Therefore, each vertex in $V'$ is in exactly two paths in $\mathcal{P}'$, and by construction, the non-$V'$ endvertices of $\cP'$ (each in some set $B_i$) are pairwise at least $\sqrt{\log n}$ apart in $G'$. Thus, picking appropriate neighbourhoods of endvertices of $\cP$ in $L$, together with $\cP'$ and $\cB'$, yields the second alternative.
\end{proof}


\begin{thebibliography}{99}

	
	\bibitem{Bol78}
	B. Bollob\'as,
	\newblock Nested Cycles in graphs. 
	\newblock \emph{Proc. Colloq. Internat. CNRS} (J. C. Bermond, J. C. Fournier, M. das Vergnas, and D. Sotteau, Eds.), Presses du CNRS, Paris, (1978). 
	
	\bibitem{CES94}
	G. Chen, P. Erd\H{o}s and W. Staton,
	\newblock Proof of a Conjecture of Bollob\'as on Nested Cycles. 
	\newblock \textit{Journal of Combinatorial Theory, Seies B}, \textbf{66}, (1994), 38--43. 
	
	\bibitem{CH63}
	K. Corradi and A. Hajnal,
	\newblock On the maximal number of independent circuits of a graph. 
	\newblock \textit{Acta Math. Acad. Sci. Hungar.}, \textbf{14}, (1963), 423--443. 
	
	\bibitem{Ega96}
	Y. Egawa,
	\newblock Vertex-disjoint cycles of the same length.
	\newblock \textit{J. Combin. Theory Ser. B}, \textbf{66}, (1996), 168--200. 
	
	\bibitem{Erdos75}
	P. Erd\H{o}s,
	\newblock Problems and results in Graph Theory and Combinatorial Analysis. 
	\newblock \textit{Proceedings Fifth British Combinatorial Conference}, (1975), 169--192.
	
	\bibitem{Hag85}
	R. H\"aggkvist,
	\newblock Equicardinal disjoint cycles in sparse graphs.
	\newblock \textit{Ann. Discrete Math.}, (1985), 269--273.
		
	\bibitem{Haj} P.~Hajnal, 
	\newblock Partition of graphs with condition on the connectivity and minimum degree, 
	\newblock \emph{Combinatorica} {\bf 3} (1983), 95--99.
	
	\bibitem{HKL20}
	J. Haslegrave, J. Kim and H. Liu,
	\newblock Extremal density for sparse minors and subdivisions. 
	\newblock \textit{International Mathematics Research Notices}, (2021), \url{https://doi.org/10.1093/imrn/rnab154}.
	
	\bibitem{KHSS17}
	J. Kim, H. Liu, M. Sharifzadeh and K. Staden,
	\newblock Proof of Koml\'os Conjecture on Hamiltonian subsets. 
	\newblock \textit{Proceedings of the London Mathematical Society}, \textbf{115}(5), (2017), 974--1013. 
	
	\bibitem{KS96}
	J. Koml\'{o}s and E. Szemer\'{e}di,
	\newblock Topological cliques in graphs II. \newblock\textit{Combinatorics, Probability and Computing}, \textbf{5}(1), (1996), 79--90.

\bibitem{KO} D.~K\"uhn and D.~Osthus, 
\newblock Partitions of graphs with high minimum degree or connectivity, 
\newblock \emph{J. Combinatorial Theory B} {\bf 88} (2003), 29--43. 
	
	\bibitem{MadLM}
	H. Liu and R.H. Montgomery,
	\newblock A proof of Mader's conjecture on large clique subdivisions in $C_4$-free graphs. 
	\newblock \textit{Journal of the London Mathematical Society}, \textbf{95}(1), (2017), 203--222.
	
	\bibitem{EH-LM}
	H. Liu and R.H. Montgomery,
	\newblock A solution to Erd\H{o}s and Hajnal's odd cycle problem, 
	\newblock\textit{preprint}, arXiv:2010.15802, 2020.
	
	\bibitem{Sti96}
	M. Stiebitz,
	\newblock Decomposing graphs under degree constraints. 
	\newblock \textit{J. Graph Theory}, \textbf{23}, (1996), 321--324.


\bibitem{Thom} C.~Thomassen, 
\newblock Graph decomposition with constraints on the connectivity and minimum degree,
\newblock \emph{J. Graph Theory} {\bf 7} (1983), 165--167.


	\bibitem{Ver02}
	J. Verstra\"ete,
	\newblock A note on vertex disjoint cycles. 
	\newblock \textit{Combinatorics, Probability and Computing}, \textbf{11}, (2002), 97--102.
	
\end{thebibliography}
\end{document}